\documentclass[12pt]{article}
\usepackage{amsfonts}
\usepackage{amsmath}
\usepackage{amsthm}
\usepackage{graphicx}

\newtheorem{theorem}{Theorem}[section]
\newtheorem{lemma}[theorem]{Lemma}
\newtheorem{definition}{Definition}[section]
\newtheorem{proposition}[theorem]{Proposition}
\numberwithin{equation}{section}

\title{Multiple Hermite polynomials and simultaneous Gaussian quadrature\thanks{%
Work supported EOS project 30889451 and FWO project G.0864.16N
}}

\author{Walter Van Assche\footnotemark[2]
        \and Anton Vuerinckx}

\begin{document}

\maketitle

\renewcommand{\thefootnote}{\fnsymbol{footnote}}

\footnotetext[2]{Department of Mathematics, KU Leuven, Celestijnenlaan 200B box 2400, BE-3001 Leuven, Belgium.}

\begin{abstract}
Multiple Hermite polynomials are an extension of the classical Hermite polynomials for which orthogonality conditions are imposed with respect to $r>1$ normal (Gaussian) weights $w_j(x)=e^{-x^2+c_jx}$ with different means $c_j/2$, $1 \leq j \leq r$. These polynomials have a number of properties, such as a Rodrigues formula,  recurrence relations (connecting polynomials with nearest neighbor multi-indices), a differential equation, etc. 
The asymptotic distribution of the (scaled) zeros is investigated and an interesting new feature happens: depending on the distance between the 
$c_j$, $1 \leq j \leq r$, the zeros may accumulate on $s$ disjoint intervals, where $1 \leq s \leq r$. 
We will use the zeros of these multiple Hermite polynomials to approximate integrals of the form $\displaystyle \int_{-\infty}^{\infty} f(x) \exp(-x^2 + c_jx)\, dx$ simultaneously for $1 \leq j \leq r$ for the case $r=3$ and the situation when the zeros accumulate on three disjoint intervals.
We also give some properties of the corresponding quadrature weights.
\end{abstract}



\section{Simultaneous Gaussian quadrature} 

Let $w_1,\ldots,w_r$ be $r \geq 1$ weight functions on $\mathbb{R}$ and $f: \mathbb{R} \to \mathbb{R}$.
Simultaneous quadrature is a numerical method to approximate $r$ integrals
\[    \int_\mathbb{R} f(x) w_j(x)\, dx, \qquad 1 \leq j \leq r  \]
by sums
\[     \sum_{k=1}^N \lambda_{k,N}^{(j)} f(x_{k,N})         \]
at the same $N$ points $\{x_{k,N}, 1 \leq k \leq N\}$ but with weights $\{\lambda_{k,N}^{(j)}, 1 \leq k \leq N\}$ which depend on $j$.
This was described by Borges \cite{Borges} in 1994 but was originally suggested by Aurel Angelescu \cite{Angel} in 1918, whose work seems
to have gone unnoticed. The past few decades it became clear that this is closely related to multiple orthogonal polynomials, in a similar
way as Gauss quadrature is related to orthogonal polynomials. The motivation in \cite{Borges} was involving a color signal $f$, 
which can be transmitted using three colors: red-green-blue (RGB).
For this we need the amount of R-G-B in $f$, given by
\[   \int_{-\infty}^\infty f(x) w_R(x)\, dx, \quad  \int_{-\infty}^\infty f(x) w_G(x)\, dx,  \quad
   \int_{-\infty}^\infty f(x) w_B(x)\, dx.   \]
A natural question is whether this can be done with $N$ function evaluations and maximum degree of accuracy.
If we choose $n$ points for each integral, and then use Gaussian quadrature, then this would require $3n$ function evaluations
for a degree of accuracy $2n-1$. A better choice is to choose the zeros of the multiple orthogonal polynomials $P_{n,n,n}$ 
for the weights $(w_R,w_G,w_B)$ and then use interpolatory quadrature. This again requires $3n$ function evaluations, but the degree of accuracy 
is increased to $4n-1$. We will call this method based on zeros of multiple orthogonal polynomials the simultaneous Gaussian quadrature method.
Some interesting research problems for simultaneous Gaussian quadrature are
\begin{itemize}
  \item To find the multiple orthogonal polynomials when the weights $w_1,\ldots,w_r$ are given.
  \item To study the location and computation of the zeros of the multiple orthogonal polynomials.
  \item To study the behavior and computation of the weights $\lambda_{k,N}^{(j)}$. 
  \item To investigate the convergence of the quadrature rules.
\end{itemize}
Part of this research has already been started in \cite{CousWVA}, \cite{FPIllanLOP}, \cite{LubWVA}, \cite{MilStan}, \cite{WVA}, but
there is still a lot to be done in this field. 

\section{Multiple orthogonal polynomials} \

\begin{definition}
Let $\mu_1,\ldots,\mu_r$ be $r$ positive measures on $\mathbb{R}$ and let $\vec{n}=(n_1,\ldots,n_r)$ be a multi-index in $\mathbb{N}^r$. The (type II) multiple orthogonal
polynomial $P_{\vec{n}}$ is the monic polynomial of degree $|\vec{n}|=n_1+n_2+\cdots+n_r$
that satisfies the orthogonality conditions
\[    \int x^k P_{\vec{n}}(x) \, d\mu_j(x)  = 0, \qquad 0 \leq k \leq n_j-1 ,  \]
for $1 \leq j \leq r$.
\end{definition}

Such a monic polynomial may not exist, or may not be unique. One needs conditions on the (moments of) the
measures $(\mu_1,\ldots,\mu_r)$. Two important cases have been introduced for which all the multiple orthogonal polynomials
exist and are unique.
The measures $(\mu_1,\ldots,\mu_r)$ are an \textit{Angelesco system} if $\textup{supp}(\mu_j) \subset \Delta_j$, where the $\Delta_j$ are intervals which are pairwise disjoint: $\Delta_i \cap \Delta_j = \emptyset$ whenever $i \neq j$.

\begin{theorem}[Angelesco]
For an Angelesco system the multiple orthogonal polynomials exist for every multi-index $\vec{n}$. Furthermore
$P_{\vec{n}}$ has $n_j$ simple zeros in each interval $\Delta_j$.
\end{theorem}

For a proof, see \cite[Ch.~4, Prop. 3.3]{NikiSor} or \cite[Thm. 23.1.3]{Ismail}.
The behavior of the quadrature weights for simultaneous Gaussian quadrature is known for this case 
(see \cite[Ch.~4 4, Prop. 3.5]{NikiSor}, \cite[Thm. 1.1]{LubWVA}).

\begin{theorem}
The quadrature weights $\lambda_{k,n}^{(j)}$ are positive for the $n_j$ zeros on $\Delta_j$. The remaining
quadrature weights have alternating sign, with those for the zeros closest to the interval $\Delta_j$ positive.
\end{theorem}

Another important case is when the measures form an AT-system.
The weight functions $(w_1,\ldots,w_r)$ are an \textit{algebraic Chebyshev system} (AT-system) on $[a,b]$ if
\begin{multline*}   w_1,xw_1,x^2w_1,\ldots,x^{n_1-1} w_1, \ w_2, xw_2, x^2 w^2, \ldots, x^{n_2-1} w_2, \ldots, \\
   w_r, xw_r, x^2w_r, \ldots, x^{n_r-1} w_r  
\end{multline*}
are a Chebyshev system on $[a,b]$ for every $(n_1,\ldots,n_r) \in \mathbb{N}^r$.

\begin{theorem}
For an AT-system the multiple orthogonal polynomials exist for every multi-index $(n_1,\ldots,n_r)$.
Furthermore $P_{\vec{n}}$ has $|\vec{n}|$ simple zeros on the interval $[a,b]$.
\end{theorem}

For a proof, see \cite[Ch.~4, Corr. of Thm. 4.3]{NikiSor} or \cite[Thm. 23.1.4]{Ismail}.

\section{Multiple Hermite polynomials}
We will consider the weight functions
\[     w_j(x) = e^{-x^2+c_jx}, \qquad x \in \mathbb{R}, \]
with real parameters $c_1,\ldots,c_j$ such that $c_i \neq c_j$ whenever $i \neq j$.
These weights are proportional to normal weights with means at $c_j/2$ and variance $\sigma^2 = \frac12$.
They form an AT-system and the corresponding multiple orthogonal polynomials are known as \textit{multiple Hermite polynomials} $H_{\vec{n}}$.
They can be obtained by using the Rodrigues formula
\[  e^{-x^2}  H_{\vec{n}}(x) = (-1)^{|\vec{n}|} 2^{-|\vec{n}|} 
     \left( \prod_{j=1}^r  e^{-c_jx} \frac{d^{n_j}}{dx^{n_j}} e^{c_jx} \right)  e^{-x^2}  ,  \] 
from which one can find the explicit expression
\begin{equation*}   H_{\vec{n}}(x) = (-1)^{|\vec{n}|} 2^{-|\vec{n}|} 
             \sum_{k_1=0}^{n_1} \cdots \sum_{k_r=0}^{n_r} \binom{n_1}{k_1} \cdots \binom{n_r}{k_r}
      c_1^{n_1-k_1} \cdots c_r^{n_r-k_r} (-1)^{|\vec{k}|} H_{|\vec{k}|}(x) .
\end{equation*}
See \cite[\S 23.5]{Ismail}.
Multiple orthogonal polynomials satisfy a system of recurrence relations connecting the nearest neighbors. For multiple Hermite polynomials
one has 
\[    xH_{\vec{n}}(x) = H_{\vec{n}+\vec{e}_k}(x) + \frac{c_k}2 H_{\vec{n}}(x) + \frac12 \sum_{j=1}^r n_j H_{\vec{n}-\vec{e}_j}(x), 
 \qquad  1 \leq k \leq r,  \]
where $\vec{e}_1=(1,0,0,\ldots,0), \vec{e}_2 = (0,1,0,\ldots,0), \ldots, \vec{e}_r = (0,0,\ldots,0,1)$. 
They also have interesting differential relations, such as $r$ raising operators
\[     \left( e^{-x^2+c_jx} H_{\vec{n}-\vec{e}_j}(x) \right)' = -2 e^{-x^2+c_jx} H_{\vec{n}}(x), \qquad 1 \leq j \leq r.  \]
and a lowering operator
\[     H_{\vec{n}}'(x) = \sum_{j=1}^r n_j H_{\vec{n}-\vec{e}_j}(x) .  \] 
See \cite[Eqs. (23.8.5)--(23.8.6)]{Ismail}.
Combining these raising operators and the lowering operator gives a differential equation of order $r+1$
\[     \left(\prod_{j=1}^r D_j \right) D H_{\vec{n}}(x) = -2 \left( \sum_{j=1}^r n_j \prod_{i \neq j} D_j \right) H_{\vec{n}}(x), \]
where
\[    D = \frac{d}{dx}, \qquad   D_j = e^{x^2-c_jx} D e^{-x^2+c_jx}  = D + (-2x+c_j) I .  \]

From now on we deal with the case $r=3$ and weights $c_1=-c, c_2=0, c_3=c$:
\[    w_1(x) = e^{-x^2-cx}, \quad  w_2(x) = e^{-x^2}, \quad    w_3(x) = e^{-x^2+cx} . \]

\subsection{Zeros}
Let $x_{1,3n} < \ldots < x_{3n,3n}$ be the zeros of $H_{n,n,n}$. 
First we will show that, for $c$ large enough, the zeros of $H_{n,n,n}$ lie on three disjoint intervals around $-c/2$, $0$ and $c/2$.

\begin{proposition}  \label{prop:3.1}
For $c$ sufficiently large (e.g., $c > 4\sqrt{4n+1}$) the zeros of $H_{n,n,n}$ are on three disjoint intervals $I_1 \cup I_2 \cup I_3$,
where
\[  I_1 = [ -\frac{c}2 - \sqrt{4n+1}, -\frac{c}2 + \sqrt{4n+1}], \quad I_2 =  [- \sqrt{4n+1}, \sqrt{4n+1}], \]
\[  I_3 =  [\frac{c}2 - \sqrt{4n+1},\frac{c}2 + \sqrt{4n+1}], \]
and each interval contains $n$ simple zeros.
\end{proposition}

\begin{proof}
Suppose $x_1,x_2,\ldots,x_m$ are the sign changes of $H_{n,n,n}$ on $I_3$ and that $m<n$. Let $\pi_m(x)=(x-x_1)(x-x_2)\cdots(x-x_m)$, then
$H_{n,n,n}(x)\pi_m(x)$ does not change sign on $I_3$. By the multiple orthogonality one has
\begin{equation}  \label{Hq0}
    \int_{-\infty}^\infty H_{n,n,n}(x)\pi_m(x) e^{-x^2+cx}\, dx = 0.
\end{equation}
Suppose that $H_{n,n,n}(x)\pi_m(x)$ is positive on $I_3$, then by the infinite-finite range inequalities 
(see, e.g., \cite[Ch. 4, Thm. 4.1]{LL}, where we take $Q(x)=x^2-cx$, $p=1$, $t=4n+1$, so that $\Delta_t = I_3$) one has
\[   \int_{\mathbb{R} \setminus I_3} |H_{n,n,n}(x)\pi_m(x)| e^{-x^2+cx}\, dx < \int_{I_3} H_{n,n,n}(x)\pi_m(x) e^{-x^2+cx}\, dx, \]
so that
\[   \int_{\mathbb{R}\setminus I_3} H_{n,n,n}(x)\pi_m(x) e^{-x^2+cx}\, dx > - \int_{I_3} H_{n,n,n}(x)\pi_m(x) e^{-x^2+cx}\, dx  \]
and
\begin{multline*}   \int_{-\infty}^\infty H_{n,n,n}(x)\pi_m(x) e^{-x^2+cx}\, dx  \\
   = \int_{I_3} H_{n,n,n}(x)\pi_m(x) e^{-x^2+cx}\, dx
     + \int_{\mathbb{R}\setminus I_3} H_{n,n,n}(x)\pi_m(x) e^{-x^2+cx}\, dx  > 0,  
\end{multline*}
which is in contradiction with \eqref{Hq0}. This mean that our assumption that $m < n$ is false and hence $m \geq n$. We can repeat the
reasoning for $I_2$ and $I_1$, and since $H_{n,n,n,}$ is a polynomial of degree $3n$, we must conclude that each interval contains $n$
zeros of $H_{n,n,n}$, which are all simple. Clearly the three intervals are disjoint when $c > 4\sqrt{4n+1}$.
\end{proof}

This result shows that, for large $c$, the multiple Hermite polynomials behave very much like an Angelesco system, i.e., multiple orthogonal polynomials
for which the orthogonality conditions are on disjoint intervals. 
 Some results for simultaneous Gauss quadrature for Angelesco systems were
proved earlier in \cite[Ch. 4, Prop. 3.4 and 3.5]{NikiSor} and \cite{LubWVA}. In this paper we will show that similar results are true for
multiple Hermite polynomials when $c$ is large.

The intervals $I_1,I_2,I_3$ are in fact a bit too large, because they were obtained
by using the infinite-finite range inequalities for one weight only, and not for the three weights simultaneously.
In order to study the zeros in more detail, we will take the parameter $c$ proportional to $\sqrt{n}$ and scale the zeros by a factor $\sqrt{n}$ as well.
This amounts to investigating the polynomials $H_{n,n,n}(\sqrt{n} x)$ with $c = \sqrt{n} \hat{c} $. In order to find for which values of $c$
the zeros are accumulating on three disjoint intervals as $n \to \infty$, we will investigate the asymptotic distribution of the zeros.
Our main theorem in this section is

\begin{theorem}   \label{thm:3.2}
There exists a $c^*>0$ such that for the zeros of $H_{n,n,n}$ with $c=\sqrt{n}\hat{c}$ one has 
\[   \lim_{n \to \infty} \frac{1}{3n}  \sum_{j=1}^{3n} f\left( \frac{x_{j,3n}}{\sqrt{n}} \right) = \int f(x) v(x)\, dx  \]
where $v$ is a probability density supported on three intervals $[-b,-a] \cup [-d,d] \cup [a,b]$  $(0 < d < a < b)$ when $\hat{c} > c^*$
and $v$ is supported on one interval $[-b,b]$ when $\hat{c}<c^*$. The numerical value is $c^*=4.10938818$.
\end{theorem}

Such a phase transition when the zeros cluster on one interval when the parameters are close together or on two intervals when the parameters
are far apart, was first observed and proved for $r=2$ by Bleher and Kuijlaars \cite{BleKuijl}.

\begin{figure}[ht]
\centering
\includegraphics[width=2.5in]{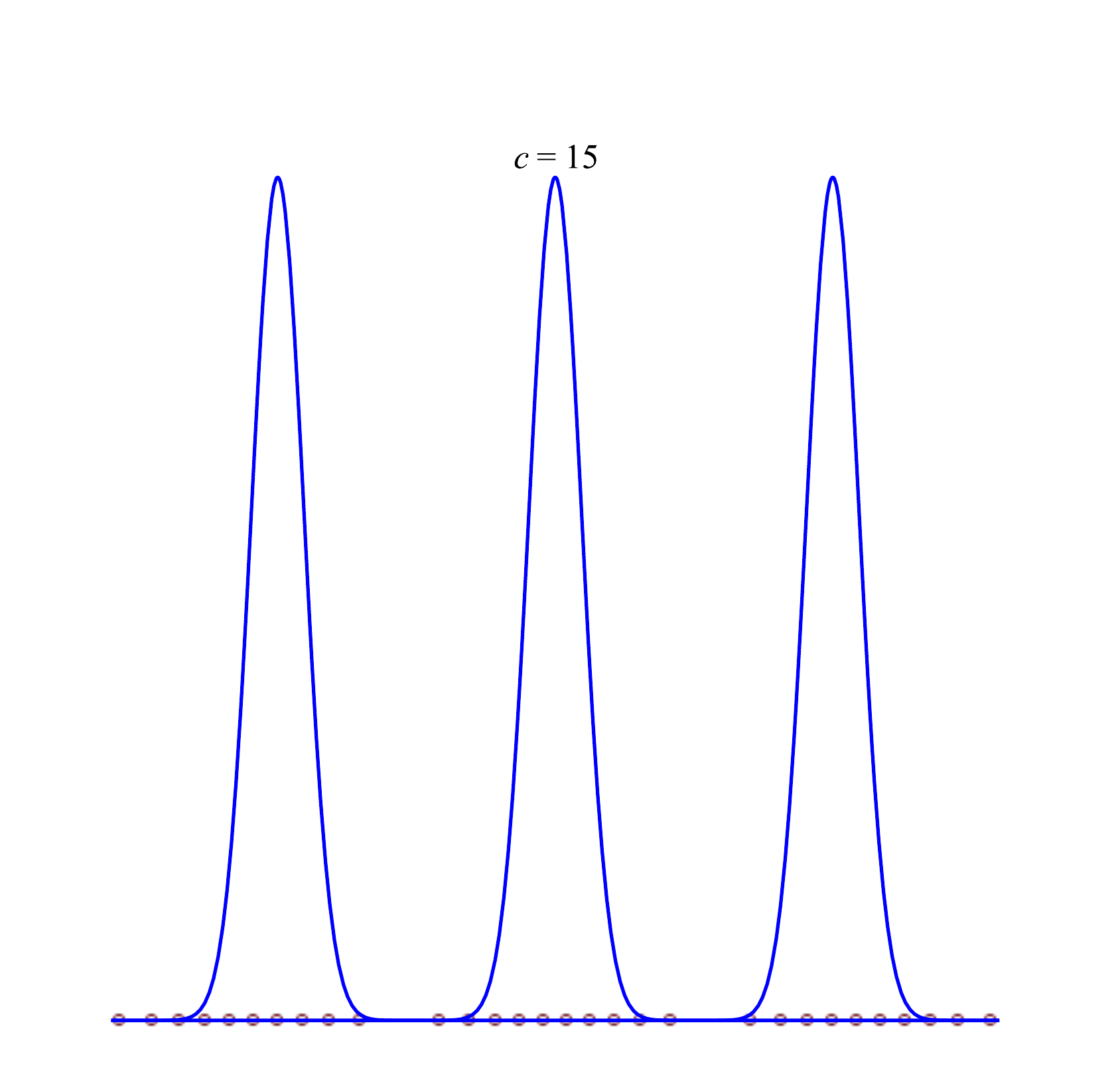}
\includegraphics[width=2.5in]{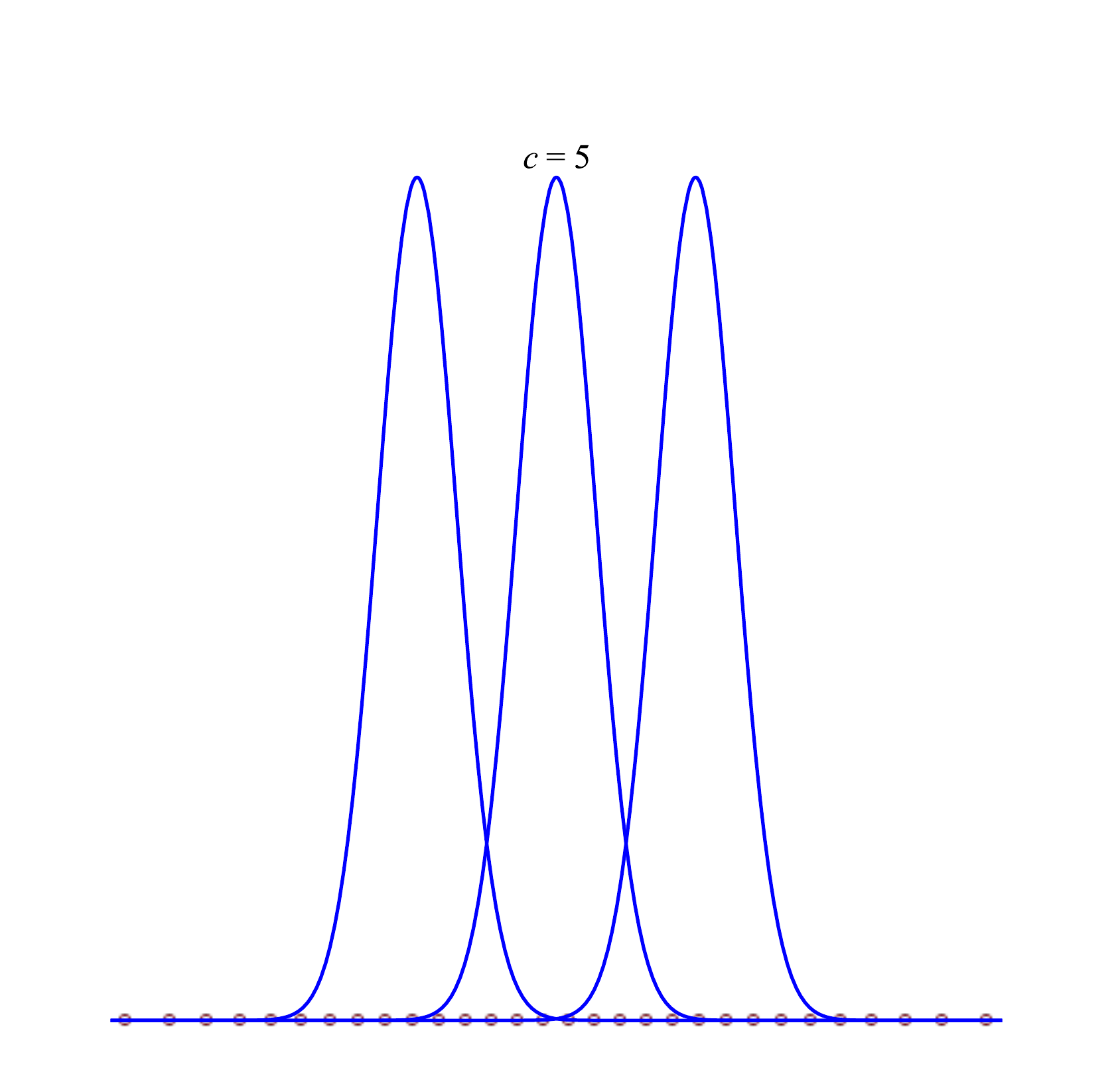}
\caption{The weight functions and the zeros of $H_{10,10,10}$ for $c=15$ (left) and $c=5$ (right).}
\label{fig:Hzeros}
\end{figure}

\begin{proof}
The differential equation for $y=H_{n,n,n}(x)$ becomes
\begin{equation*}  
y'''' -6x y''' + (12x^2-c^2-6)y'' +[-8x^3+(2c^2+12)x]y' 
  = -2n [ 3y'' -12xy' + (12x^2 -c^2-6)y] .
\end{equation*}
The scaling amounts to studying zeros of $H_{n,n,n}(\sqrt{n} x)$ and these are multiple orthogonal polynomials for the weight functions
\[   w_1(x) = e^{-n(x^2+\hat{c}x)}, \quad  w_2(x) = e^{-nx^2}, \quad   w_3(x) = e^{-n(x^2-\hat{c}x)} .   \]
Consider the rational function
\[    S_n(z) = \frac{1}{\sqrt{n}} \frac{H_{n,n,n}'(\sqrt{n} z)}{H_{n,n,n}(\sqrt{n} z)} 
= \frac{1}{n}\sum_{j=1}^{3n} \frac{1}{z-\frac{x_{j,3n}}{\sqrt{n}}} = \int \frac{d\mu_n(x)}{z-x} , \]
where $\mu_n$ is the discrete measure with mass $1/n$ at each scaled zero $x_{j,3n}/\sqrt{n}$:
\[   \mu_n = \frac{1}{n} \sum_{j=1}^{3n}  \delta_{x_{j,3n}/\sqrt{n}}.  \]
The sequence $(S_n)_{n \in \mathbb{N}}$ is a family of analytic functions which is uniformly bounded on every compact subset of 
$\mathbb{C} \setminus \mathbb{R}$,
hence by Montel's theorem there exists a subsequence $(S_{n_k})_k$ that converges
uniformly on compact subsets of $\mathbb{C} \setminus \mathbb{R}$ to an analytic function $S$, and also its derivatives
converge uniformly on these compact subsets:
\[   S_{n_k} \to S, \quad S_{n_k}' \to S', \quad S_{n_k}'' \to S'', \quad  S_{n_k}''' \to S''' .  \]
Since each $S_n$ is a Stieltjes transform of a positive measure (with total mass 3), the limit is of the form
\[    S(z) = 3 \int \frac{d\mu(x)}{z-x}\, dx , \]
with $\mu$ a probability measure on $\mathbb{R}$, which describes the asymptotic distribution of the scaled zeros, and
$\mu_n$ converges weakly to the measure $3\mu$ alongs the chosen subsequence. 
This function $S$ may depend on the selected subsequence $(n_k)_k$, but we will show
that every convergent subsequence has the same limit $S$. Observe that
\[ H_{n,n,n}'(\sqrt{n}z) = \sqrt{n} S_n H_{n,n,n}(\sqrt{n}z), \]
from which we can find
\begin{align*}
   H_{n,n,n}''(\sqrt{n}z) &= (S_n' + n S_n^2) H_{n,n,n}(\sqrt{n}z), \\
   H_{n,n,n}'''(\sqrt{n}z) &= \frac{1}{\sqrt{n}} (S_n'' + 3n S_n'S_n + n^2 S_n^3) H_{n,n,n}(\sqrt{n}z), \\
   H_{n,n,n}''''(\sqrt{n}z) &= \frac{1}{n} (S_n''' + 4n S_n''S_n + 3n (S_n')^2 + 6n^2 S_n^2 S_n' + n^3 S_n^4) H_{n,n,n}(\sqrt{n}z).
\end{align*}
Put this in the differential equation (with $x=\sqrt{n}z$ and $c=\sqrt{n}\hat{c}$), then as $n=n_k \to \infty$ one finds
\begin{equation}  \label{EqS}
    S^4 - 6z S^3 + (12z^2-\hat{c}^2+6) S^2 + (-8z^3+2\hat{c}^2z-24z)S +2 (12z^2-\hat{c}^2) = 0 .  
\end{equation}
This is an algebraic equation of order 4 and hence it has four solutions $S_{(1)}, S_{(2)}, S_{(3)}, S_{(4)}$. A careful analysis
of these solutions and equation \eqref{EqS} near infinity shows that for $z \to \infty$
\[   S_{(1)}(z) = \frac{3}{z} + \mathcal{O}(\frac{1}{z^2}), \quad   S_{(2)}(z) = 2z+\hat{c}+ \mathcal{O}(\frac{1}{z}), \]
\[   S_{(3)}(z) = 2z + \mathcal{O}(\frac{1}{z}), \quad   S_{(4)}(z) = 2z-\hat{c}+ \mathcal{O}(\frac{1}{z}). \]
We are therefore interested in $S_{(1)}(z)$ since it gives the required Stieltjes transform
\[    S_{(1)}(z) = S(z) =  3 \int \frac{d\mu(x)}{z-x}\, dx .  \]
The algebraic equation is independent of the selected subsequence, which implies that every subsequence $(S_{n_k})_k$ has the same
limit, which in turn implies that the full sequence $(S_n)_n$ converges to this limit $S$.
The measure $\mu$ can be retrieved by using the Stieltjes-Perron inversion theorem
\[   \mu((a,b)) + \frac12 \mu(\{a\}) + \mu(\{b\}) = \lim_{\epsilon \to 0+} \frac{1}{2\pi i} \int_a^b \frac{S(x-i\epsilon)-S(x+i\epsilon)}{3}\, dx. \]
If $\mu$ has no mass points, then the density $v$ of $\mu$ is given by
\[    v(x) = \frac{1}{2\pi i}  \lim_{\epsilon \to 0+} \frac{S(x-i\epsilon)-S(x+i\epsilon)}{3}, \]
hence the support of the density $v$ is given by the set on $\mathbb{R}$ where $S$ has a jump discontinuity.
This can be analyzed by investigating the {discriminant} of the algebraic equation: 
\begin{multline}  \label{pol}
    256 \hat{c}^6 z^6 - 128 \hat{c}^4 (\hat{c}^4+18 \hat{c}^2-18) z^4 \\
  +\ 16 \hat{c}^2 (\hat{c}^8+12 \hat{c}^6+240 \hat{c}^4-1008 \hat{c}^2+432) z^2 
 - 32 \hat{c}^2(\hat{c}^2+4\hat{c}+6)^2 (\hat{c}^2-4\hat{c}+6)^2 .
\end{multline}
This is a polynomial of degree $6$ in the variable $z$. The support of $v$ is where this polynomial is negative.
There is a phase transition from one interval to three intervals when the $z$-polynomial \eqref{pol} has two
double roots. 
\begin{figure}[ht]
\centering
\includegraphics[width=1.6in]{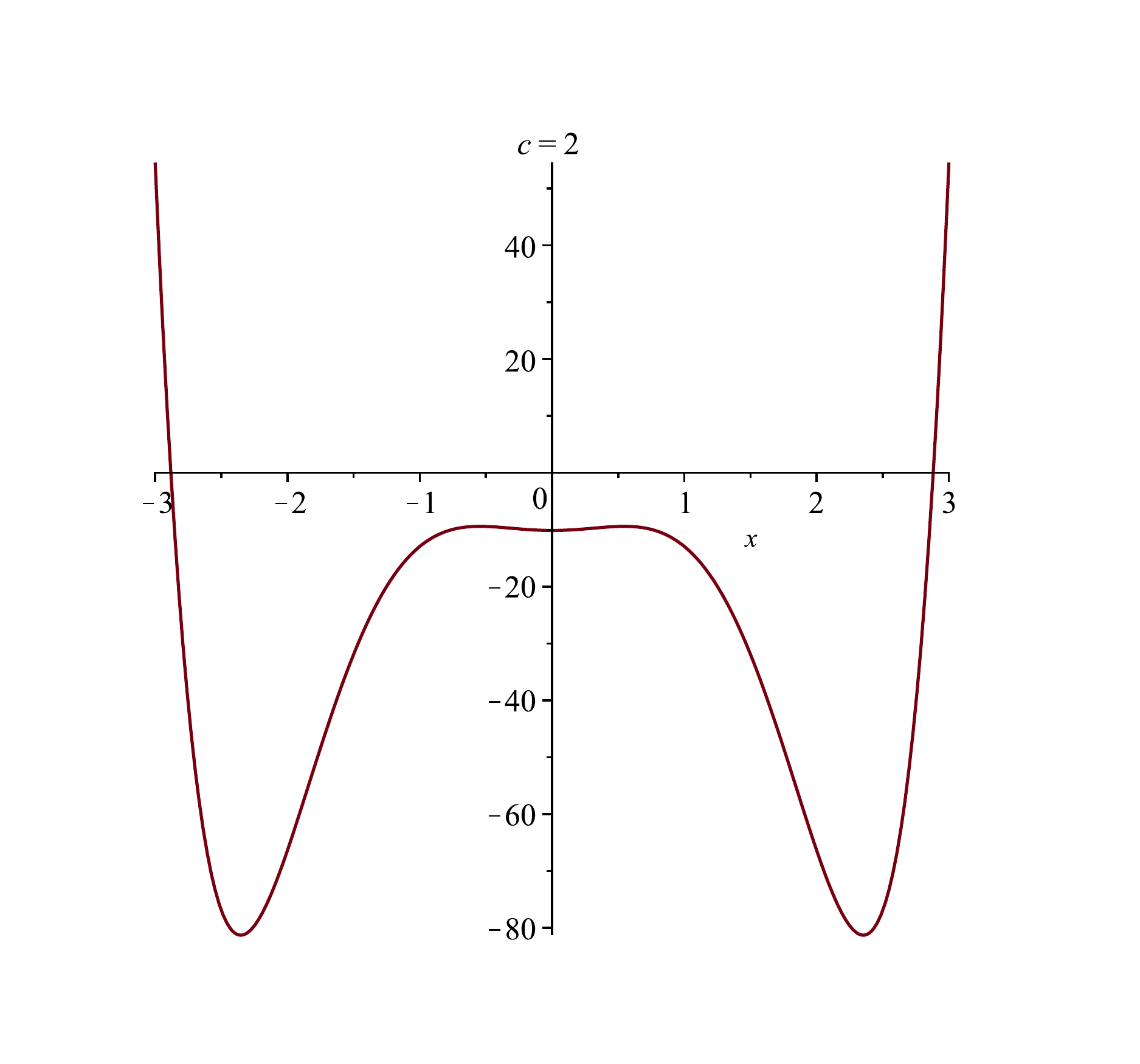}
\includegraphics[width=1.6in]{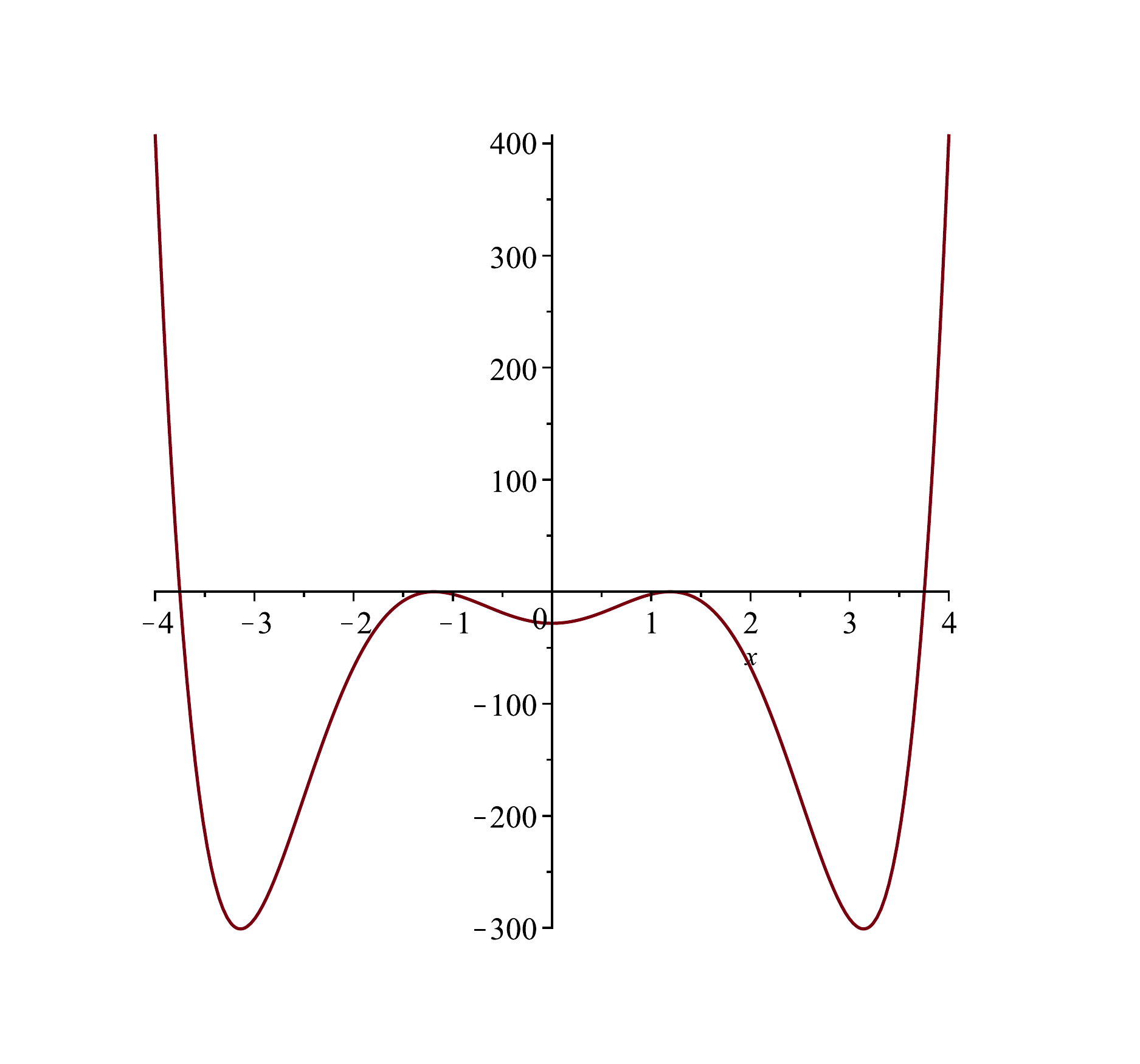}
\includegraphics[width=1.6in]{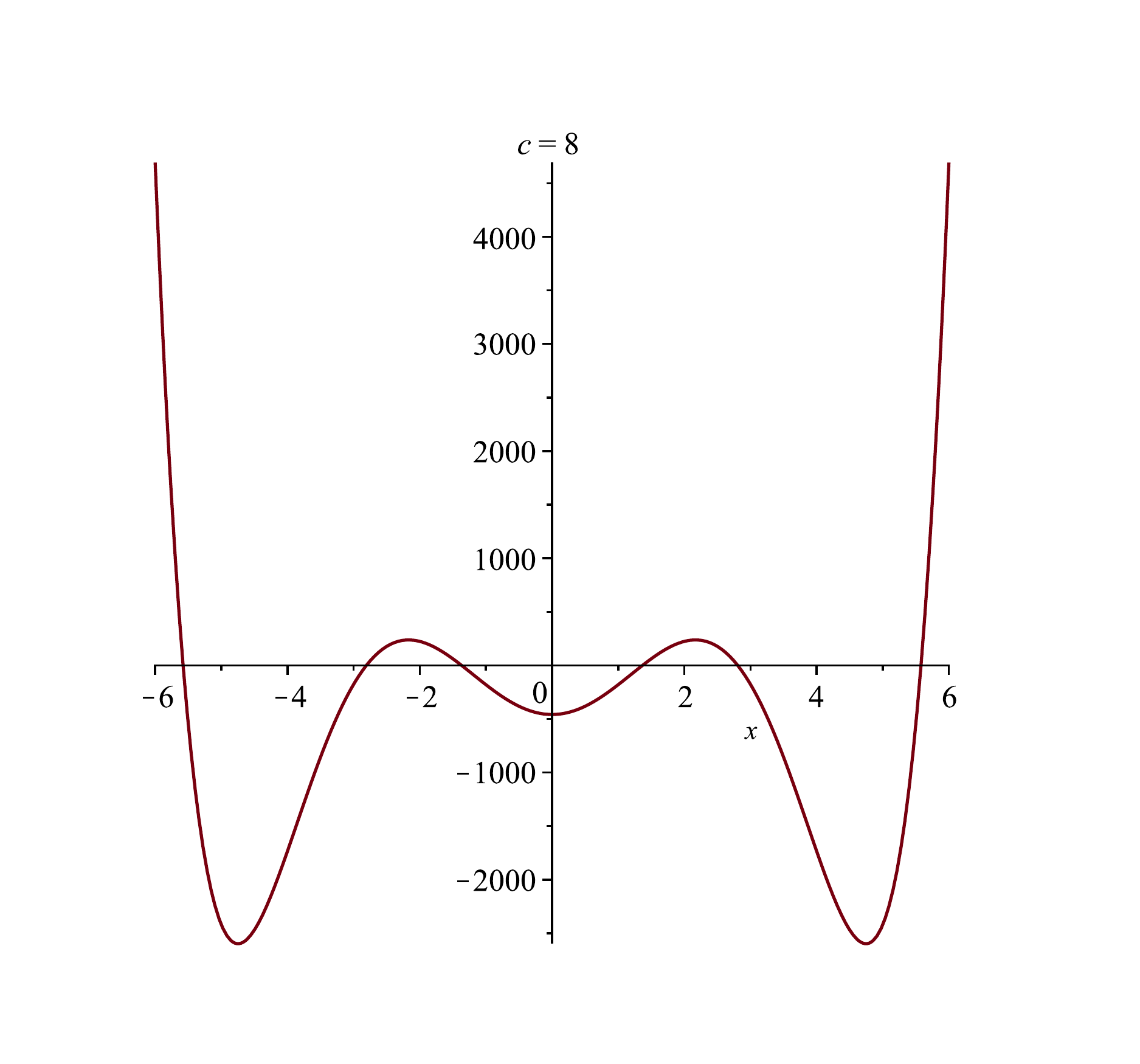}
\caption{The polynomial \eqref{pol} for $\hat{c}=2$ (left), $\hat{c}=\hat{c}^*$ (middle) and $\hat{c}=8$ (right).}
\label{fig:polyS}
\end{figure}
This happens when the discriminant of the $z$-polynomial \eqref{pol} is zero
\[     (\hat{c}^2-4\hat{c}+6)^2 \hat{c}^{32} (\hat{c}^2+2)^4 (\hat{c}^2+4\hat{c}+6)^2 (\hat{c}^6- \frac{27}{2} 
   \hat{c}^4   -54 \hat{c}^2-54)^6 = 0.  \]
The only positive real zero is the positive real root of
\[   \hat{c}^6- \frac{27}{2} \hat{c}^4   -54 \hat{c}^2-54 = 0  \]
and this is $c^* = 4.10938818$. 
\end{proof}

Observe that the phase transition $c^*$ is at a smaller value than the one suggested by Proposition \ref{prop:3.1}, which would give the
value $8$. As mentioned before, this is because in Proposition \ref{prop:3.1} we used the infinite-finite range inequalities for one single
weight and not of the three weights simultaneously.

\section{Some potential theory}
From now one we assume $\hat{c} > c^* = 4.10938818$.
The Stieltjes transform of the asymptotic zero distribution is
\[   3\int \frac{v(x)}{z-x}\, dx =  S(z) = \int_{-b}^{-a} \frac{d\nu_1(x)}{z-x} + \int_{-d}^d \frac{d\nu_2(x)}{z-x} + \int_{a}^{b} \frac{d\nu_3(x)}{z-x}. \]
The measures $\nu_1,\nu_2,\nu_3$ are unit measures that are minimizing the expression
\[    \sum_{i=1}^3 \sum_{j=1}^3  c_{i,j} I(\mu_i,\mu_j)  +  \sum_{i=1}^3 \int V_i(x)\, d\mu_i(x)   \]
over all unit measures $\mu_1,\mu_2,\mu_3$ supported on $\mathbb{R}$, with
\[    I(\mu_i,\mu_j) = \iint  \log \frac{1}{|x-y|}\, d\mu_i(x)\,d\mu_j(x), \qquad
 C = (c_{i,j}) = \begin{pmatrix}  1 & 1/2 & 1/2 \\  1/2 & 1 & 1/2  \\ 1/2 & 1/2 & 1  \end{pmatrix}  \]
and 
\[   V_1(x) = x^2+\hat{c}x, \quad   V_2(x) = x^2 , \quad   V_3(x) =  x^2-\hat{c}x .    \]
This is the vector equilibrium problem for an Angelesco system \cite[Ch. 5, \S 6]{NikiSor}.
Define the logarithmic potential
\[     U(x;\mu) = \int \log \frac{1}{|x-y|} \, d\mu(y).  \]
The variational conditions for this {vector equilibrium problem} are
\begin{align}   
      2 U(x;\nu_1) + U(x;\nu_2) + U(x;\nu_3) + V_1(x) &= \ell_1, \qquad   x \in [-b,-a],  \label{Uvar1}  \\
      2 U(x;\nu_1) + U(x;\nu_2) + U(x;\nu_3) + V_1(x) &\geq \ell_1, \qquad   x \in \mathbb{R} \setminus [-b,-a], 
\end{align}
\begin{align} 
      U(x;\nu_1) + 2U(x;\nu_2) + U(x;\nu_3) + V_2(x) &= \ell_2, \qquad   x \in [-d,d],  \label{Uvar3} \\
      U(x;\nu_1) + 2U(x;\nu_2) + U(x;\nu_3) + V_2(x) &\geq \ell_2, \qquad   x \in \mathbb{R} \setminus [-d,d], 
\end{align}
\begin{align}  
      U(x;\nu_1) + U(x;\nu_2) + 2U(x;\nu_3) + V_3(x) &= \ell_3, \qquad   x \in [a,b], \label{Uvar5} \\
      U(x;\nu_1) + U(x;\nu_2) + 2U(x;\nu_3) + V_3(x) &\geq \ell_3, \qquad   x \in \mathbb{R} \setminus [a,b]. \label{Uvar6}
\end{align}
where $\ell_1,\ell_2,\ell_3$ are constants (Lagrange multipliers). As an example, we have plotted these functions
in Figure \ref{fig:UV} for $\hat{c}=6$.
\begin{figure}[ht]
\centering
\includegraphics[width=5in]{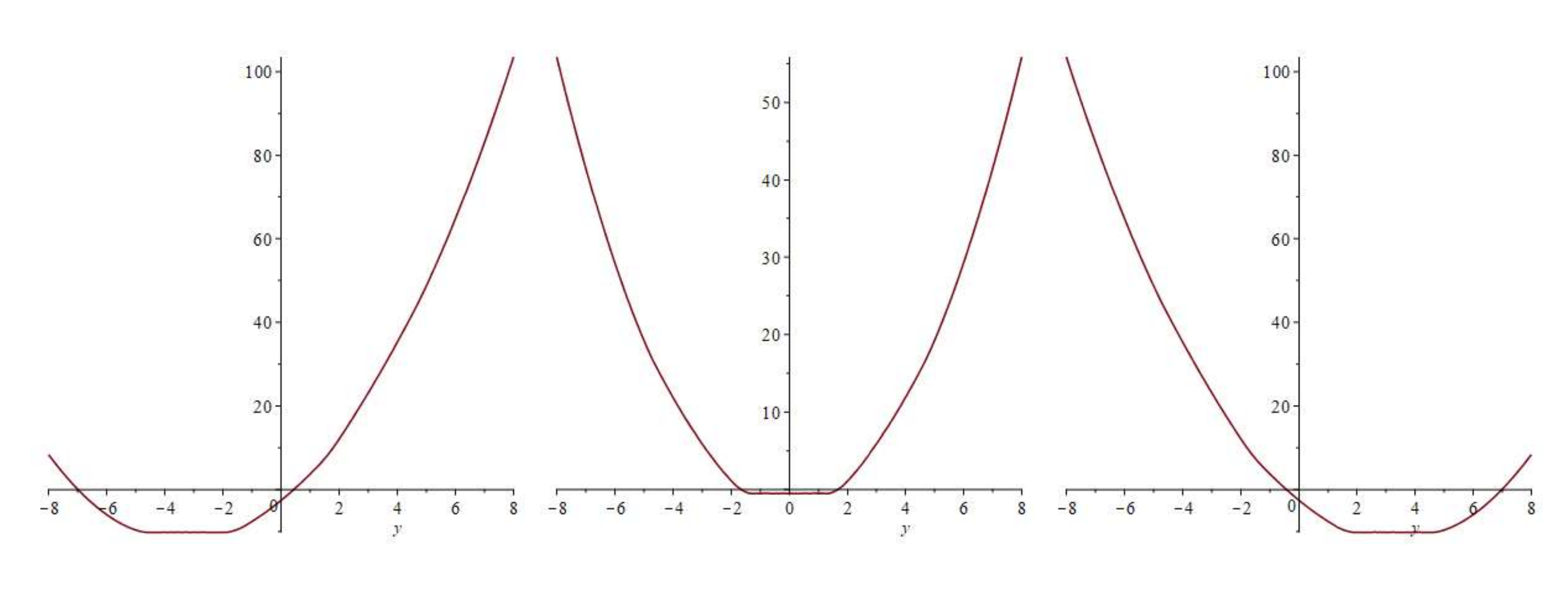}
\caption{$2U(x;\nu_1)+U(x;\nu_2)+U(x;\nu_3)+V_1(x)$ (left), $U(x;\nu_1)+2U(x;\nu_2)+U(x;\nu_3)+V_2(x)$ (middle), 
$U(x;\nu_1)+U(x;\nu_2)+2U(x;\nu_3)+V_3(x)$ (right).}
\label{fig:UV}
\end{figure}
The measures $\nu_1,\nu_2,\nu_3$ give the asymptotic distribution of the (scaled) zeros of $H_{n,n,n}$ on the intervals $[-b,-a]$, $[-d,d]$ and
$[a,b]$ respectively. They are absolutely continuous and their densities can be found as the jumps of an algebraic function $\xi$ on the real line.
The function $\xi$ satisfies the algebraic equation
\[    \xi^4 -2z\hat{c}\xi^3 + (6-\hat{c}^2) \xi^2 + 2\hat{c}^2 z \xi - 2 \hat{c}^2 = 0, \]
which has four solutions $\xi_1,\xi_2,\xi_3,\xi_4$ which behave near infinity as
\[  \xi_1(z) = 2z- \frac{3}{z} + \mathcal{O}(\frac{1}{z^2}), \quad \xi_2(z) = -\hat{c} + \frac{1}{z} + \mathcal{O}(\frac{1}{z^2}), \]  
\[  \xi_3(z) = \frac{1}{z} + \mathcal{O}(\frac{1}{z^2}), \quad \xi_4(z) = \hat{c} + \frac{1}{z} + \mathcal{O}(\frac{1}{z^2}). \]
The densities $\nu_1',\nu_2',\nu_3'$ are given by
\[  \nu_1'(x) = - \frac{(\xi_2)_+(x)-(\xi_2)_-(x)}{2\pi i}, \quad \nu_2'(x) = - \frac{(\xi_3)_+(x)-(\xi_3)_-(x)}{2\pi i}, \] 
\[     \nu_3'(x) = - \frac{(\xi_4)_+(x)-(\xi_4)_-(x)}{2\pi i}.  \]
The relation between the algebraic function $S$ from \eqref{EqS} is given by
\[     S = \frac{2}{\xi} + \frac{2}{\xi+\hat{c}} + \frac{2}{\xi-\hat{c}} .  \]
The Stieltjes transforms of $\nu_1,\nu_2,\nu_3$ are related to the solutions of \eqref{EqS} by
\[  S_{(1)}(z) = \int_{-b}^{-a} \frac{d\nu_1(x)}{z-x} + \int_{-d}^{d} \frac{d\nu_2(x)}{z-x} + \int_a^b \frac{d\nu_3(x)}{z-x},
    \quad  S_{(3)}(z) = 2z - \int_{-d}^d \frac{d\nu_2(x)}{z-x},   \]
\[  S_{(2)}(z) = 2z+\hat{c}-\int_{-b}^{-a} \frac{d\nu_1(x)}{z-x},   
   \quad   S_{(4)}(z) = 2z - \hat{c} - \int_a^b \frac{d\nu_3(x)}{z-x}   .   \] 
 
\section{The quadrature weights}
Recall that for polynomials $f$ of degree $\leq 4n-1$
\begin{eqnarray}  
     \int_{-\infty}^\infty  f(x) e^{-n(x^2+\hat{c}x)}\, dx &=& \sum_{k=1}^{3n} \lambda_{k,3n}^{(1)} f(x_{k,3n}) ,   \label{quad1} \\
     \int_{-\infty}^\infty  f(x) e^{-nx^2}\, dx &=& \sum_{k=1}^{3n} \lambda_{k,3n}^{(2)} f(x_{k,3n}) , \label{quad2}  \\
     \int_{-\infty}^\infty  f(x) e^{-n(x^2-\hat{c}x)}\, dx &=& \sum_{k=1}^{3n} \lambda_{k,3n}^{(3)} f(x_{k,3n}) . \label{quad3}
\end{eqnarray}
Here $x_{k,3n}$ are the zeros of $H_{n,n,n}(x) = p_n(x)q_n(x)r_n(x)$, where 
$p_n$ has its zeros on $[-b,-a]$, $q_n$ on  $[-d,d]$, and $r_n$ on $[a,b]$.  
Take $f(x) = \pi_{2n-1}(x)q_n(x)r_n(x)$, with $\pi_{2n-1}$ of degree $\leq 2n-1$, then \eqref{quad1} gives
\[     \int_{-\infty}^\infty  \pi_{2n-1}(x) q_n(x)r_n(x) e^{-n(x^2+\hat{c}x)}\, dx 
   = \sum_{k=1}^{n} \lambda_{k,3n}^{(1)} q_n(x_{k})r_n(x_{k})\pi_{2n-1}(x_{k}) .  \]
This is the Gauss quadrature formula for the weight function $q_n(x)r_n(x) e^{-n(x^2+\hat{c}x)}$ with quadrature nodes
at the zeros of $p_n$. So we have

\begin{lemma}  \label{lem:5.1}
The first $n$ quadrature weights for the first integral \eqref{quad1} are
\[   \lambda_{k,3n}^{(1)} q_n(x_{k,3n})r_n(x_{k,3n}) = \lambda_{k,n}(q_nr_n\,d\mu_1), \qquad  1 \leq k \leq n, \]
where $\lambda_{k,n}(q_nr_n\,d\mu_1)$ are the usual Christoffel numbers of Gaussian quadrature for the weight $q_n(x)p_n(x)e^{-n(x^2+\hat{c}x)}$ on
$\mathbb{R}$.  
\end{lemma}

For the middle $n$ quadrature weights and the last $n$ quadrature weights we have a weaker statement. By taking 
$f(x) = \pi_{n-1}(x)p_n^2(x)r_n(x)$, with $\pi_{n-1}$ of degree $\leq n-1$, the quadrature formula \eqref{quad1} gives
\[     \int_{-\infty}^\infty  \pi_{n-1}(x) p_n^2(x)r_n(x) e^{-n(x^2+\hat{c}x)}\, dx 
   = \sum_{k=n+1}^{2n} \lambda_{k,3n}^{(1)} p_n^2(x_{k})r_n(x_{k})\pi_{n-1}(x_{k}) .  \]
This is not a Gauss quadrature rule but the Lagrange interpolatory rule for the weight function $p_n^2(x)r_n(x) e^{-n(x^2+\hat{c}x)}$,
with quadrature nodes at the zeros of $q_n$. So now we have

\begin{lemma}   \label{lem:5.2}
The middle $n$ quadrature weights for the first integral are
\[   \lambda_{k,3n}^{(1)} p_n^2(x_{k,3n})r_n(x_{k,3n}) = w_{k,n}(q_n), \qquad n+1 \leq k \leq 2n, \]
where $w_{k,n}(q_n)$ are the quadrature weights for the Lagrange interpolatory quadrature at the zeros of $q_n$ and weight function
$p_n^2(x)r_n(x) e^{-n(x^2+\hat{c}x)}$. 
\end{lemma}

In a similar way, we take $f(x) = \pi_{n-1}(x)p_n^2(x)q_n(x)$, with $\pi_{n-1}$ of degree $\leq n-1$, so that \eqref{quad1} becomes 
\[    \int_{-\infty}^\infty  \pi_{n-1}(x) p_n^2(x)q_n(x) e^{-n(x^2+\hat{c}x)}\, dx 
   = \sum_{k=2n+1}^{3n} \lambda_{k,3n}^{(1)} p_n^2(x_{k})q_n(x_{k})\pi_{n-1}(x_{k}) .  \]
We then have
 
\begin{lemma}    \label{lem:5.3}
The last $n$ quadrature weights for the first integral are
\[   \lambda_{k,3n}^{(1)} p_n^2(x_{k,3n})q_n(x_{k,3n}) = w_{k,n}(r_n), \qquad 2n+1 \leq k \leq 3n, \]
where $w_{k,n}(r_n)$ are the quadrature weights for the Lagrange interpolatory quadrature at the zeros of $r_n$ and weight function
$p_n^2(x)q_n(x) e^{-n(x^2+\hat{c}x)}$. 
\end{lemma}

Of course similar results are true for the quadrature weights $\lambda_{k,3n}^{(2)}$ for the second integral \eqref{quad2} and
the quadrature weights $\lambda_{k,3n}^{(3)}$ for the third integral \eqref{quad3}.

The weight function $q_n(x)r_n(x) e^{-n(x^2+\hat{c}x)}$ is not a positive weight on the whole real line, but it is positive on $[-b,-a]$
since the zeros of $q_n$ and $r_n$ are on $[-d,d]$ and $[a,b]$ respectively, at least when $n$ is large. We can prove the
following result.
 
\begin{theorem}  \label{thm:5.4}
Let $\hat{c}$ be sufficiently large\footnote{$\hat{c} > 8$ certainly works, but we conjecture that $\hat{c} > c^*$ is sufficient.}. 
For the quadrature weights of the first integral \eqref{quad1} one has
\[  \lambda_{k,3n}^{(1)} > 0,  \qquad 1 \leq k \leq n, \]
and
\[  \textup{sign } \lambda_{k,3n}^{(1)} = (-1)^{k-n+1}, \qquad  n+1 \leq k \leq 3n.  \]
\end{theorem}

\begin{proof}
For the first $n$ weights we use $f(x)=p_n^2(x)q_n(x)r_n(x)/(x-x_{k,3n})^2$ in \eqref{quad1} to find  (we write $x_k = x_{k,3n}$)
\[   \lambda_{k,3n}^{(1)} [p_n'(x_k)]^2 q_n(x_k) r_n(x_k) = \int_{-\infty}^\infty \frac{p_n^2(x)}{(x-x_k)^2} q_n(x)r_n(x) e^{-n(x^2+\hat{c}x)}\, dx. \]
Clearly $[p_n'(x_k)]^2 q_n(x_k) r_n(x_k) >0$ since $x_k \in [-b,-a]$ and the zeros of $q_n$ and $r_n$ are on $[-d,d]$ and $[a,b]$ respectively.
So we need to prove that the integral is positive. Let $I_1= [-\frac{\hat{c}}{2} - \sqrt{4+1/n}, -\frac{\hat{c}}{2} + \sqrt{4+1/n}]$, then by
Proposition \ref{prop:3.1} all the zeros of $p_n$ are in $I_1$ and hence $[-b,-a] \subset I_1$. 
If $\hat{c}$ is large enough, then $q_nr_n$ is positive on $I_1$ and by the infinite-finite range inequality (see Proposition \ref{prop:3.1})
\[    \int_{\mathbb{R} \setminus I_1}  \frac{p_n^2(x)}{(x-x_k)^2} |q_n(x)r_n(x)| e^{-n(x^2+\hat{c}x)}\, dx 
    >  \int_{I_1} \frac{p_n^2(x)}{(x-x_k)^2} q_n(x)r_n(x) e^{-n(x^2+\hat{c}x)}\, dx, \]
so that
\begin{multline*}
   \int_{-\infty}^\infty \frac{p_n^2(x)}{(x-x_k)^2} q_n(x)r_n(x) e^{-n(x^2+\hat{c}x)}\, dx \\
    = \int_{I_1} \frac{p_n^2(x)q_n(x)r_n(x)}{(x-x_k)^2}  e^{-n(x^2+\hat{c}x)}\, dx
       + \int_{\mathbb{R} \setminus I_1} \frac{p_n^2(x)q_n(x)r_n(x)}{(x-x_k)^2}  e^{-n(x^2+\hat{c}x)}\, dx > 0. 
\end{multline*}

For the middle $n$ quadrature weights we use Lemma \ref{lem:5.2}. Clearly $p_n^2(x_k) >0$ and $\textup{sign } r_n(x_k) = (-1)^n$ since all the zeros
of $r_n$ are on $[a,b]$ and $x_k \in [-d,d]$ for $n+1 \leq k \leq 2n$. Furthermore for the Lagrange quadrature nodes one has  
\[     w_{k,n}(q_n) = \int_{-\infty}^\infty \frac{q_n(x)}{(x-x_k) {q_n'(x_k)} } p_n^2(x) r_n(x) e^{-n(x^2+\hat{c}x)}\, dx  \]
where $\textup{sign } q_n'(x_k) = (-1)^{k-2n}$. Observe that for large enough $\hat{c}$ one has $\textup{sign } q_n(x)/(x-x_k) = (-1)^{n-1}$
on $I_1$ since all the zeros of $q_n$ are on $[-d,d]$, and also $\textup{sign }r_n(x) = (-1)^n$ on $I_1$ since all the zeros of $r_n$
are on $[a,b]$.
By the infinite-finite range inequality one has
\[   \int_{\mathbb{R} \setminus I_1} \frac{|q_n(x)|}{|x-x_k|} p_n^2(x)|r_n(x)| e^{-n(x^2+\hat{c}x)}\, dx
      < - \int_{I_1} \frac{q_n(x)}{x-x_k} p_n^2(x)r_n(x) e^{-n(x^2+\hat{c}x)}\, dx   \]
so that
\[  \int_{-\infty}^\infty \frac{q_n(x)}{(x-x_k)} p_n^2(x) r_n(x) e^{-n(x^2+\hat{c}x)}\, dx  < 0 . \]
This gives $\textup{sign } \lambda_{k,3n}^{(1)} = (-1)^{k-n+1}$ for $n+1 \leq k \leq 2n$. In a similar way one finds the sign
of $\lambda_{k,3n}^{(1)}$ for $2n+1 \leq k \leq 3n$ by using Lemma \ref{lem:5.3}.
\end{proof}  

\begin{figure}[ht]
\centering
\includegraphics[width=3in]{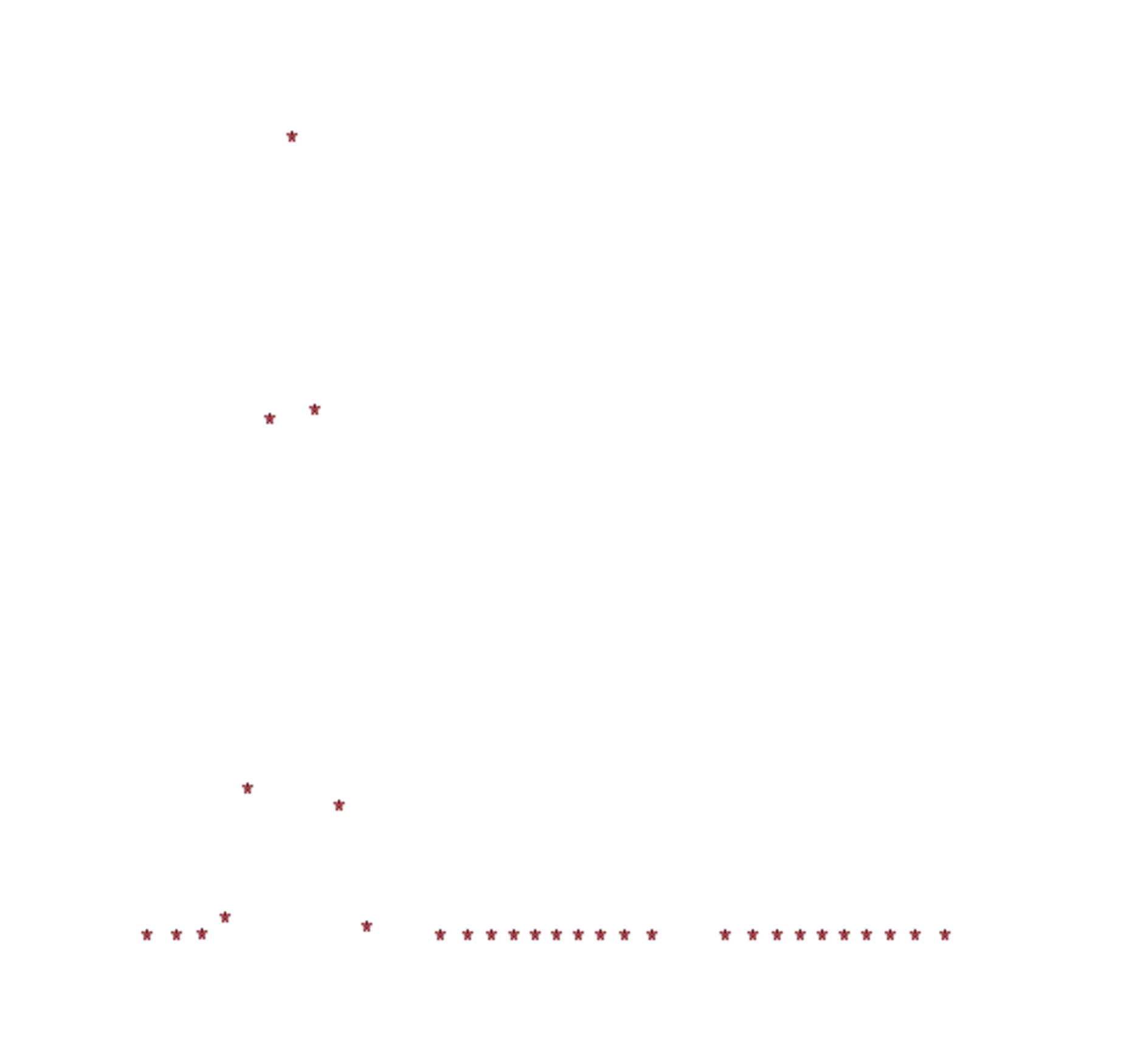}
\caption{The quadrature weights $\lambda_{k,30}^{(1)}$ for the first integral ($\hat{c}=4.7434$).}
\label{fig:lambda}
\end{figure}

For the quadrature weights $\lambda_{k,3n}^{(2)}$ one has a similar result, which we state without proof.

\begin{theorem}  \label{thm:5.5}
Let $\hat{c}$ be sufficiently large. 
For the quadrature weights of the second integral \eqref{quad2} one has
\[  \lambda_{k,3n}^{(2)} > 0,  \qquad n+1 \leq k \leq 2n, \]
and
\[  \textup{sign } \lambda_{k,3n}^{(2)} = \begin{cases}
    (-1)^{k-n}, &  1 \leq k \leq n, \\
    (-1)^{k+1}, & 2n+1 \leq k \leq 3n.
                    \end{cases}       \]
\end{theorem}
Observe that the quadrature weights for the nodes outside $[-d,d]$ are alternating, but the weights for the nodes closest to $[-d,d]$ are positive.

For the quadrature nodes $\lambda_{k,3n}^{(3)}$ one has

\begin{theorem}  \label{thm:5.6}
Let $\hat{c}$ be sufficiently large. 
For the quadrature weights of the third integral \eqref{quad3} one has
\[  \lambda_{k,3n}^{(3)} > 0,  \qquad 2n+1 \leq k \leq 3n, \]
and
\[  \textup{sign } \lambda_{k,3n}^{(3)} = (-1)^{k}, \qquad  1 \leq k \leq 2n. \]
\end{theorem}

Having positive quadrature weights is a nice property, as is well known for Gauss quadrature. The alternating quadrature weights are not so nice,
but we can show that they are exponentially small.

\begin{theorem}  \label{thm:5.7} 
Suppose $\hat{c}$ is sufficiently large (see the footnote in Theorem \ref{thm:5.4}).
For the positive quadrature weights one has
\begin{equation}  \label{poslam1}
   \limsup_{n \to \infty} \left( \lambda_{k,3n}^{(1)} \right)^{1/n} \leq e^{-V_1(x)} ,
\end{equation}
whenever $x_k \to x \in (-b,-a)$. For the quadrature weights with alternating sign
\begin{equation}  \label{limsuplambda}
   \limsup_{n \to \infty} |\lambda_{k,3n}^{(1)}|^{1/n} \leq \exp \left( 2 U(x;\nu_1) + U(x;\nu_2) + U(x;\nu_3) - \ell_1 \right)  
\end{equation}
whenever $x_{k,3n} \to x \in (-d,d) \cup (a,b)$.
\end{theorem}

\begin{proof}
Let $x \in (-b,-a)$, then we use Lemma \ref{lem:5.1} to see that $\lambda_{k,3n}^{(1)} q_n(x_k) r_n(x_k) = \lambda_{k,n}$,
where $\lambda_{k,n}$ are the Gauss quadrature weights for the weight function $q_n(x)r_n(x) e^{-nV_1(x)}$. 
We can use the Chebyshev-Markov-Stieltjes inequalities \cite[\S 3.41]{Szego} for the Gauss quadrature weights to find
\[    \lambda_{k,3n}^{(1)} q_n(x_k) r_n(x_k) \leq \int_{x_{k-1}}^{x_{k+1}} q_n(x)r_n(x) e^{-n V_1(x)}\, dx.  \]
By the mean value theorem, we have
\[   \int_{x_{k-1}}^{x_{k+1}} q_n(x)r_n(x) e^{-n(x^2+\hat{c}x)}\, dx = (x_{k+1}-x_{k-1}) q_n(\xi_n) r_n(\xi_n) e^{-n V_1(\xi_n)} , \]
for some $\xi_n \in (x_{k-1},x_{k+1})$. Then, since $x_{k+1}-x_{k-1} \leq b-a$, we find
\[   \limsup_{n \to \infty} \left( \lambda_{k,3n}^{(1)} \right)^{1/n} \leq e^{-V_1(x)}, \]
whenever $x_k \to x \in (-b,-a)$, since
\[    \lim_{n \to \infty} |q_n(x_k)|^{1/n} = \exp \bigl( -U(x;\nu_2) \bigr) =  \lim_{n \to \infty} |q_n(\xi_n)|^{1/n}, \]
and
\[    \lim_{n \to \infty} |r_n(x_k)|^{1/n} = \exp \bigl( -U(x;\nu_3) \bigr) =  \lim_{n \to \infty} |r_n(\xi_n)|^{1/n}. \]

Let $x \in (-d,d)$, then we use Lemma \ref{lem:5.2} to find
\[    |\lambda_{k,3n}^{(1)} | = \frac{1}{p_n^2(x_k) |r_n(x_k)| |q_n'(x_k)|} \left|\int_{-\infty}^\infty \frac{q_n(x)}{x-x_k} p_n^2(x) r_n(x) e^{-n(x^2+\hat{c}x)}\, dx
   \right|.  \]
For the polynomials $p_n$ and $r_n$ one has
\[   \lim_{n \to \infty} |p_n(x)|^{1/n} = \exp \bigl(-U(x;\nu_1) \bigr), \quad  \lim_{n \to \infty} |r_n(x)| = \exp \bigl( -U(x;\nu_3) \bigr), \]
uniformly in $x \in [-d,d]$, which already gives
\[   \lim_{n \to \infty} \frac{1}{p_n^2(x_k) |r_n(x_k)|} = \exp \bigl( 2U(x;\nu_1)+U(x;\mu_3) \bigr) . \]
For the integral we use the infinite-finite range inequality (see Proposition \ref{prop:3.1}) to find
\[     \left| \int_{-\infty}^\infty \frac{q_n(x)}{x-x_k} p_n^2(x) r_n(x) e^{-n(x^2+\hat{c}x)}\, dx \right|
   \leq 2 \int_{-I_1} \frac{|q_n(x)|}{|x-x_k|} p_n^2(x) |r_n(x)| e^{-n(x^2+\hat{c}x)}\, dx  . \]
For $\hat{c}$ sufficiently large the intervals $I_1$, $[-d,d]$ and $[a,b]$ are disjoint, hence for $x \in I_1$ and $x_k \in [-d,d]$ 
we have $|x-x_k| > \textup{dist}(I_1,[-d,d])=\delta_1$.
We thus have
\[  \left| \int_{-\infty}^\infty \frac{q_n(x)}{x-x_k} p_n^2(x) r_n(x) e^{-n(x^2+\hat{c}x)}\, dx \right|
  \leq \frac{2}{\delta_1} \int_{I_1} |q_n(x)| p_n^2(x) |r_n(x)| e^{-n(x^2+\hat{c}x)}\, dx . \]
Observe that the integrand is
\[    p_n^2(x) |q_n(x) r_n(x)| e^{-n(x^2+\hat{c}x)} = \exp \Bigl(-n \bigl( 2U(x;\nu_1) + U(x;\nu_2) + U(x;\nu_3) +V_1(x) \bigr)\Bigr), \]
and as $n \to \infty$ the $n$th root thus converges to $\exp(-\ell_1)$ when $x \in [-b,-a]$ or $\leq \exp(-\ell_1)$ when $x \notin [-b,-a]$.
We thus have (see the third Corollary \cite[p.~199]{NikiSor} for an Angelesco system)
\[    \limsup_{n\to\infty} \left( \frac{2}{\delta_1} \int_{I_1} |q_n(x)| p_n^2(x) |r_n(x)| e^{-n(x^2+\hat{c}x)}\, dx \right)^{1/n}
     \leq e^{-\ell_1}.  \]

The $n$th root behavior of $|q_n'(x_k)|$ is more difficult because we evaluate $q_n'$ at a point in $(-d,d)$, which is on the support of $\nu_2$ where
the zeros of $q_n$ are dense. Clearly $q_n'$ has $n-1$ zeros between the zeros of $q_n$ and the asymptotic distribution of the zeros of
$q_n'$ is the same as that of $q_n$, hence 
$|q_n'(x)|^{1/n}$ converges to $\exp \bigl( -U(x;\nu_2) \bigr)$ whenever $x \notin [-d,d]$. When $x_k \to x \in (-d,d)$ one can use the principle
of descent \cite[Thm. 6.8 in Ch.~I]{ST} to find
\begin{equation}  \label{q'up}
  \limsup_{n \to \infty}  |q_n'(x_k)|^{1/n} \leq \exp \bigl( -U(x;\nu_2) \bigr), \qquad  x \in (-d,d).  
\end{equation}
To prove the inequality in the other direction, we look at the quadrature weights $\lambda_{k,3n}^{(2)}$ for the second integral \eqref{quad2} corresponding to
the nodes on $[-d,d]$ (the zeros of $q_n$). These nodes are positive and related to the Gauss quadrature nodes for the orthogonal polynomials
with weight function $p_n(x)r_n(x) e^{-nx^2}$, see Theorem \ref{thm:5.5}. 
The result corresponding to \eqref{poslam1} is
\[   \limsup_{n \to \infty} \left( \lambda_{k,3n}^{(2)} \right)^{1/n} \leq e^{-V_2(x)}, \] 
On the other hand, by taking $f(x)=p_n(x)q_n^2(x)r_n(x)/(x-x_k)^2$ in \eqref{quad2}, we see that the quadrature weight $\lambda_{k,3n}^{(2)}$ 
satisfies
\[   \lambda_{k,3n}^{(2)} p_n(x_k)r_n(x_k) [q_n'(x_k)]^2 = \int_{-\infty}^\infty \frac{p_n(x)q_n^2(x)r_n(x)}{(x-x_k)^2} e^{-nx^2}\, dx . \]
Observe that the sign of $r_n(x)$ on $[-d,d]$ is $(-1)^n$, hence
by the infinite-finite range inequalities, one finds
\[  (-1)^n \lambda_{k,3n}^{(2)} p_n(x_k)r_n(x_k) [q_n'(x_k)]^2 = (1+ r_n) \int_{I_2} \frac{p_n(x)q_n^2(x)r_n(x)}{(x-x_k)^2} e^{-nx^2}\, dx,  \]
where $|r_n| < 1$. On $I_2$ we have that $|x-x_k|\leq \delta_2$ where $\delta_2$ is the length of $I_2$, hence
\[  (-1)^n \lambda_{k,3n}^{(2)} p_n(x_k)r_n(x_k) [q_n'(x_k)]^2 \geq \frac{1+r_n}{\delta_2^2}
   \int_{I_2} |p_n(x)| q_n^2(x) |r_n(x)| e^{-nx^2}\, dx, \]
from which we find
\[  |q_n'(x_k)|^2 \geq \frac{1+r_n}{\delta_2^2} \frac{1}{\lambda_{k,3n}^{(2)} |p_n(x_k)r_n(x_k)|} \int_{I_2}  |p_n(x)| q_n^2(x) |r_n(x)| e^{-nx^2}\, dx. \]
By taking the $n$th root and by using the same reasoning as before, we thus find
\[  \liminf_{n \to \infty} |q_n'(x_k)|^{2/n} \geq \exp \bigl( V_2(x) +U(x;\nu_1)+U(x;\nu_3) -\ell_2 \bigr)  \]
and since $x \in (-d,d)$, it follows from \eqref{Uvar3} that the right hand side is $\exp \bigl( -2U(x;\nu_2)\bigr)$. 
Combined with \eqref{q'up} we then have
\[   \lim_{n \to \infty}  |q_n'(x_k)|^{1/n}  = e^{-U(x;\nu_2)},  \]
whenever $x_k \to x \in (-d,d)$.
Combining all these results gives \eqref{limsuplambda} for $x_{k,3n} \to x \in (-d,d)$. The proof for $x_{k,3n} \to x \in (a,b)$ is similar,
using Lemma \ref{lem:5.3}.
\end{proof}

The results corresponding the quadrature weights for the second integral \eqref{quad2} and the thirs integral \eqref{quad3} are:
\begin{theorem}  \label{thm:5.8} 
Suppose $\hat{c}$ is sufficiently large (see the footnote in Theorem \ref{thm:5.4}).
For the positive quadrature weights one has
\begin{equation}  \label{poslam2}
   \limsup_{n \to \infty} \left( \lambda_{k,3n}^{(2)} \right)^{1/n} \leq e^{-V_2(x)} ,
\end{equation}
whenever $x_k \to x \in (-d,d)$. For the quadrature weights with alternating sign
\begin{equation}  \label{limsuplambda2}
   \limsup_{n \to \infty} |\lambda_{k,3n}^{(2)}|^{1/n} \leq \exp \left( U(x;\nu_1) + 2U(x;\nu_2) + U(x;\nu_3) - \ell_2 \right)  
\end{equation}
whenever $x_{k,3n} \to x \in (-b,-a) \cup (a,b)$.
\end{theorem}

\begin{theorem}  \label{thm:5.9} 
Suppose $\hat{c}$ is sufficiently large (see the footnote in Theorem \ref{thm:5.4}).
For the positive quadrature weights one has
\begin{equation}  \label{poslam3}
   \limsup_{n \to \infty} \left( \lambda_{k,3n}^{(3)} \right)^{1/n} \leq e^{-V_3(x)} ,
\end{equation}
whenever $x_k \to x \in (a,b)$. For the quadrature weights with alternating sign
\begin{equation}  \label{limsuplambda3}
   \limsup_{n \to \infty} |\lambda_{k,3n}^{(3)}|^{1/n} \leq \exp \left( U(x;\nu_1) + U(x;\nu_2) + 2U(x;\nu_3) - \ell_3 \right)  
\end{equation}
whenever $x_{k,3n} \to x \in (-b,-a) \cup (-d,d)$.
\end{theorem}

\begin{figure}[ht]
\centering
\includegraphics[width=5in]{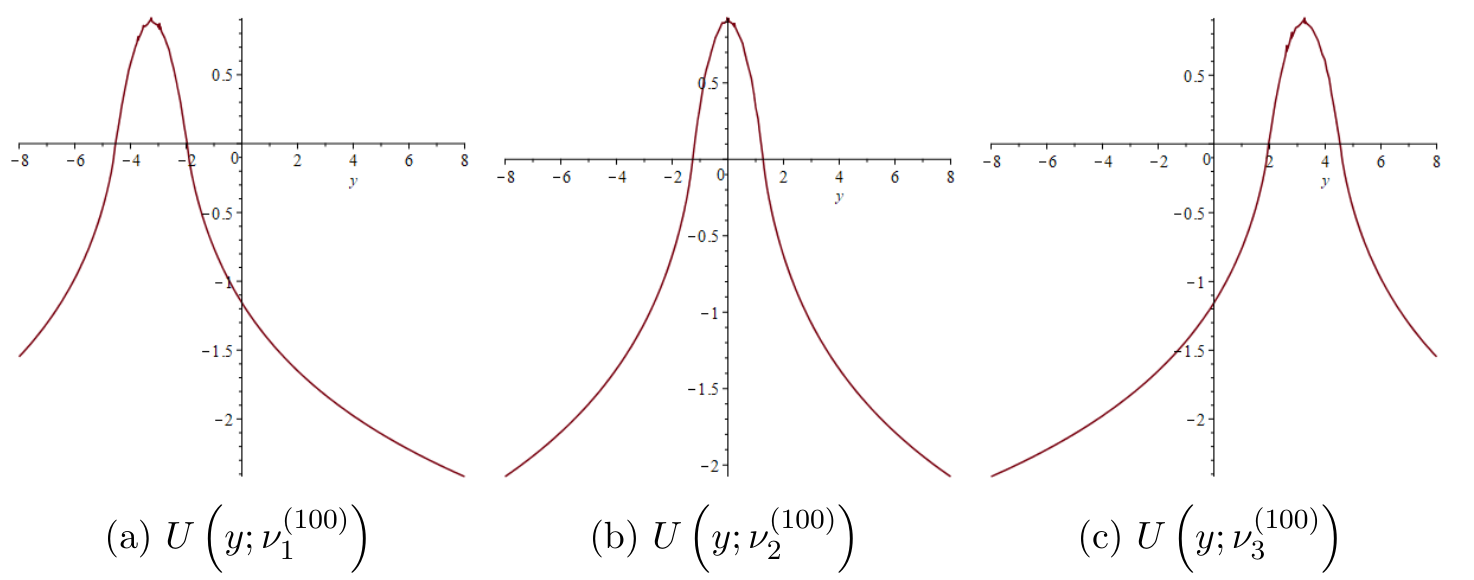}
\caption{The potentials $U(x;\nu_1)$, $U(x;\nu_2)$, $U(x;\nu_3)$ approximated by using the zeros of $H_{100,100,100}$.}
\label{fig:pot}
\end{figure}

All the upper bounds in Theorems \ref{thm:5.7}--\ref{thm:5.9} depend on the logarithmic potentials $U(x;\nu_1)$, $U(x;\nu_2), U(x;\nu_3)$,
and in particular on the linear combination of them that appears in the variational conditions \eqref{Uvar1}--\eqref{Uvar6}.
Observe that by combining \eqref{Uvar3} with \eqref{limsuplambda} we find
\[  \limsup_{n \to \infty} |\lambda_{k,3n}^{(1)}|^{1/n} \leq  e^{-V_2(x)} \exp \bigl( \ell_2 - U(x;\nu_2) - \ell_1 + U(x;\nu_1) \bigr), \]
whenever $x_k \to x \in (-d,d)$, and by using \eqref{Uvar5} we find
\[  \limsup_{n \to \infty} |\lambda_{k,3n}^{(1)}|^{1/n} \leq e^{-V_3(x)} \exp \bigl( \ell_3 - U(x;\nu_3) - \ell_1 + U(x;\nu_1) \bigr), \]
whenever $x_k \to x \in (a,b)$. Hence on $(-d,d)$ the quadrature weights are bounded from above by $e^{-V_2(x)}$ times a factor which is small,
since $\ell_2 - U(x;\nu_2) - \ell_1 + U(x;\nu_1)<0$ on $[-d,d]$. On $(a,b)$ the quadrature weights are bounded by $e^{-V_3(x)}$ times
an even smaller factor, since $\ell_3=\ell_1$ (by symmetry) and $U(x;\nu_3) > U(x;\nu_2) > U(x;\nu_1)$ for $x \in (a,b)$,
see Figure \ref{fig:pot}. This makes the alternating quadrature weights exponentially small.

\section{Numerical example}
In Table \ref{example} and Figure \ref{fig:lambda} we give the quadrature weights $\lambda_{k,3n}^{(1)}$ for the zeros of $H_{10,10,10}$
with $c=15$, which after scaling by $\sqrt{10}$ corresponds to $\hat{c}=4.734$.

This clearly shows that the first 10 zeros are positive and the remaining 20 zeros are alternating in sign and very small in absolute value.
The zeros and the quadrature weights behave in a similar way as in an Angelesco system (see \cite{LubWVA}) when $\hat{c}$ is sufficiently large.
Our scaling and the use of the weight functions 
\[    w_1(x) = e^{-n(x^2+\hat{c}x)}, \quad  w_2(x) = e^{-nx^2}, \quad    w_3(x) = e^{-n(x^2-\hat{c}c)x)}, \]
means that we are using the densities of normal distributions with means $-\hat{c}/2, 0, \hat{c}/2$ and variance $\sigma^2=1/2n$. In such case we can ignore the alternating weights and only use the positive quadrature weights $\{\lambda_{k,3n}^{(1)}: 1 \leq k \leq 3n\}$ to approximate the first
integral \eqref{quad1}. In a similar way, when we approximate the second integral \eqref{quad2} we can ignore the alternating weights and
only use the positive weights $\{ \lambda_{k,3n}^{(2)} : n+1 \leq k \leq 2n\}$, and for approximating the third integral, one can only use
$\{ \lambda_{k,3n}^{(3)}: 2n+1 \leq k \leq 3n\}$.  

\begin{table}[t]
\centering
\caption{The quadrature weights $\lambda_{k,30}^{(1)}$ for the first integral ($\hat{c}=4.7434$).}
\begin{tabular}{|l|l|}
\hline
$k$ & $\lambda_{k,30}^{(1)}$ \\
\hline
1 & $6.887653865\ 10^{-9}$ \\
2 & $4.384111578\ 10^{-6}$ \\
3 & $0.3591983034\ 10^{-3}$ \\
4 & $0.8149617619\ 10^{-2}$ \\
5 & $0.6836500666\ 10^{-1}$ \\
6 & $0.2410330694$ \\
7 & $0.3725933960$ \\
8 & $0.2452710131$ \\
9 & $0.6041135610\ 10^{-1}$ \\
10 & $0.3809098858\ 10^{-2}$ \\
11 & $6.755525278\ 10^{-6}$ \\
12 & $-5.189883715\ 10^{-6}$ \\
13 & $3.848392520\ 10^{-6}$ \\
14 & $-2.434636570\ 10^{-6}$ \\
15 & $1.261797315\ 10^{-6}$ \\
\hline
\end{tabular} \quad
\begin{tabular}{|l|l|}
\hline
$k$ & $\lambda_{k,30}^{(1)}$ \\
\hline
16 & $-5.203778435\ 10^{-7}$ \\
17 & $1.650403141\ 10^{-7}$ \\
18 & $-3.822820686\ 10^{-8}$ \\
19 & $5.890634594\ 10^{-9}$  \\
20 & $-4.840551012\ 10^{-10}$ \\
21 & $1.105332527\ 10^{-11}$ \\
22 & $-7.562667367\ 10^{-12}$ \\
23 & $3.793214538\ 10^{-12}$ \\
24 & $-1.400104912\ 10^{-12}$ \\
25 & $3.767415857\ 10^{-13}$ \\
26 & $-7.193039657\ 10^{-14}$ \\
27 & $9.260146442\ 10^{-15}$ \\
28 & $-7.331498520\ 10^{-16}$ \\
29 & $2.977117925\ 10^{-17}$ \\
30 & $-3.903292274\ 10^{-19}$ \\
\hline
\end{tabular}
\label{example}
\end{table}

\begin{verbatim}
Walter Van Assche, Anton Vuerinckx
Department of Mathematics
KU Leuven
Celestijnenlaan 200B box 2400
BE-3001 Leuven
BELGIUM
walter.vanassche@kuleuven.be
\end{verbatim}

\begin{thebibliography}{33}
\bibitem{Angel} 
A. Angelesco,
\textit{Sur l'approximation simultan\'ee de plusieurs int\'egrals d\'efinies},
C.R. Acad. Sci. Paris 167 (1918), 629--631.
\bibitem{BleKuijl}
P. Bleher, A.B.J. Kuijlaars, 
\textit{Large $n$ limit of Gaussian random matrices with external source. I}, 
Comm. Math. Phys. 252 (2004), no. 1--3, 43--76.
\bibitem{Borges}
C.F. Borges, 
\textit{On a class of Gauss-like quadrature rules}, 
Numer. Math. 67 (1994), no. 3, 271--288.
\bibitem{CousWVA} 
J. Coussement, W. Van Assche,
\textit{Gaussian quadrature for multiple orthogonal polynomials},
J. Comput. Appl. Math. 178 (2005), no.~1--2, 131--145.
\bibitem{FPIllanLOP}
U. Fidalgo Prieto, J. Ill\'an, G. L\'opez Lagomasino,
\textit{Hermite-Pad\'e approximation and simultaneous quadrature formulas},
J. Approx. Theory 126 (2004), no.~2, 171--197.
\bibitem{Ismail}
M.E.H. Ismail, 
\textit{Classical and Quantum Orthogonal Polynomials in One Variable}, 
Encyclopedia of Mathematics and its Applications 98, Cambridge University Press, 2005.
\bibitem{LL} E. Levin, D.S. Lubinsky,
\textit{Orthogonal Polynomials for Exponential Weights},
CMS Books in Mathematics, Springer, New York, 2001.
\bibitem{LubWVA} 
D.S. Lubinsky, W. Van Assche, 
\textit{Simultaneous Gaussian quadrature for Angelesco systems}, 
Ja\'en J. Approx. 8 (2016), no. 2, 113--149.
\bibitem{MilStan}
G. Milovanovi\'c, M. Stani\'c, 
\textit{Construction of multiple orthogonal polynomials by discretized Stieltjes-Gautschi procedure and
corresponding Gaussian quadratures}, 
Facta Univ. Ser. Math. Inform. 18 (2003), 9--29.
\bibitem{NikiSor}
E.M. Nikishin, V.N. Sorokin,
\textit{Rational Approximations and Orthogonality},
Translations of Mathematical Monographs 92, Amer. Math. Soc., Providence RI,1991.
\bibitem{ST} E.B. Saff, V. Totik,
\textit{Logarithmic Potentials with External Fields},
Grundlehren der mathematischen Wissenschaften, vol.~316,
Springer-Verlag, Berlin, 1997. 
\bibitem{Szego} G. Szeg\H{o},
\textit{Orthogonal Polynomials},
Amer. Math. Soc. Colloq. Publ., vol 23, Amer. Math. Soc., Providence RI, 1939 (fourth edition 1975).
\bibitem{WVA}
W. Van Assche,
\textit{Pad\'e and Hermite-Pad\'e approximation and orthogonality},
Surv. Approx. Theory 2 (2006), 61--91.
\end{thebibliography}
\end{document}